\documentclass[a4paper,11pt]{article}
\usepackage{amsmath}
\usepackage{amsfonts}
\usepackage{amssymb}
\usepackage{amsthm}
 \begin{document}
\title{VALUATION ON A FILTERED MODULE AND RELATIONS}
\author{M.H. Anjom SHoa\footnote{Mohammad Hassan Anjom SHoa, University of Biirjand, anjomshoamh@birjand.ac.ir, anjomshoamh@yahoo.com}, M.H. Hosseini\footnote{Mohammad Hossein Hosseini, University of Birjand, mhhosseini@birjand.ac.ir}}
\date{}
\maketitle
\newtheorem{Def}{Definition}[section]
\newtheorem{lem}{Lemma}[section]
\newtheorem{theo}{Theorem}[section]
\newtheorem{pro}{Proposition}[section]
\newtheorem{cro}{Corollary}[section]
\newtheorem{rem}{Remark}[section]
\begin{abstract}
In this paper we show if $R$ is a filtered ring and $M$ a
filtered $R$\_module then we can define a valuation on a module
for $M$. Then we show that we can find an skeleton of valuation on
$M$, and we prove some properties such that derived form it for a
filtered module.
\end{abstract}

\section{Introduction}
 \paragraph*{}In algebra valuation module and filtered $R$\_module are two most important structures. We know that filtered $R$\_module is  the most important structure since filtered module is a base for graded module especially associated graded module and completion and some similar results(\cite{1},\cite{2},\cite{3},\cite{7},\cite{8}). So, as these important structures, the relation between these structure is useful for finding some new structures, and if $M$ is a valuation module then $M$ has many properties that have many usage for example,  Rees valuations and asymptotic primes of rational powers in Noetherian rings and lattices(\cite{4},\cite{10}).
 \paragraph*{} In this article we investigate the relation between filtered $R$\_module and valuation module. We prove that if we have filtered $R$\_module then we can find a valuation $R$\_module on it. For this we define $\nu :M\to\mathbb{Z}$ such that  for every $t \in M$, and by lemma(\ref{sec:lemv31}), lemma(\ref{sec:lemv32}), lemma(\ref{sec:lemv33}), lemma(\ref{sec:lemv34}) and theorem(\ref{sec:thv1}) we show $\nu$ has all properties of valuation on $R$-module $M$. Also we show if $M$ is a filtered $R$\_module then it has a skeleton of valuation, continuously we prove some properties for $M$ that derived from skeleton of valuation(\cite{6},\cite{9}).

\section{Preliminaries}
\begin{Def}\label{sec:defv1}
  A filtered ring $R$ is a ring together with a family $\left\{R_{n} \right\}_{n\ge 0} $ of  additive subgroups of $R$ satisfying in the following
  conditions:
  \begin{enumerate}
  \item[i)] $R_{0} =R$;
  \item[ii)]$R_{n+1} \subseteq R_{n} $ for all $n\ge 0$;
  \item[iii)]$R_{n} R_{m} \subseteq R_{n+m} $ for all  $n,m\ge 0$.
    \end{enumerate}
 \end{Def}

 \begin{Def}\label{sec:defv2}
 Let $R$ be a ring together with a family $\left\{R_{n} \right\}_{n\ge 0} $ of  additive subgroups of $R$ satisfying the following conditions:
 \begin{enumerate}

\item[i)]$R_{0} =R$;
\item[ii)] $R_{n+1} \subseteq R_{n} $ for all $n\ge 0$;
 \item[iii)]$R_{n} R_{m} =R_{n+m} $  for all  $n,m\ge 0$,

Then we say  $R$ has a strong filtration.
 \end{enumerate}
\end{Def}
 \begin{Def}\label{sec:defv3}
 Let $R$ be a filtered ring with filtration $\left\{R_{n} \right\}_{n\ge 0} $ and $M$ be a $R$\_module with family $\left\{M_{n} \right\}_{n\ge 0} $ of subgroups of $M$ satisfying the following
 conditions:
 \begin{enumerate}
 \item[i)]$M_{0} =M$;
  \item[ii)] $M_{n+1} \subseteq M_{n} $ for all $n\ge 0$;
  \item[iii)]$R_{n} M_{m} \subseteq M_{n+m} $  for all  $n,m\ge 0$,
   \end{enumerate}
  Then $M$ is called filtered  $R$\_module.
 \end{Def}
 \begin{Def}\label{sec:dev4}
 Let $R$ be a filtered ring with filtration $\left\{R_{n} \right\}_{n\ge 0} $ and $M$ be a $R$\_module together with a family $\left\{M_{n} \right\}_{n\ge 0} $ of  subgroups of $M$ satisfying the following conditions:
  \begin{enumerate}

 \item[i)]$M_{0} =R$;
 \item[ii)] $M_{n+1} \subseteq M_{n} $ for all $n\ge 0$;
 \item[iii)]$R_{n} M_{m} =M_{n+m} $  for all  $n,m\ge 0$,
\end{enumerate}
Then we say  $M$ has a strong filtration.
 \end{Def}
 \begin{Def}\label{sec:defv5}
 Let $M$ be an $R$\_module where $R$ is a ring, and $\Delta$ an ordered set with
maximum element $\infty$ and $\Delta\neq\{\infty\}$. A mapping
$\upsilon$ of $M$ onto $\Delta$ is called a \textsl{valuation} on
$M$, if the following conditions are satisfied:
\begin{enumerate}
 \item[i)] For any $x,y\in M$, $\upsilon(x+y)\geq min\{\upsilon(x),\upsilon(y)\}$;
\item[ii)] If $\upsilon(x)\leq \upsilon(y)$, $x,y \in M$, then $\upsilon(ax)\leq
\upsilon(ay)$ for all $a\in R$;
\item[iii)] Put $\upsilon^{-1}(\infty):=\{x \in M | \upsilon(x)=\infty\}$.
If $\upsilon(az)\leq \upsilon(bz)$, where $a,b\in R$, and $z\in
M\setminus\upsilon^{-1}(\infty)$, then $\upsilon(ax)\leq
\upsilon(bx)$ for all $x\in M$
\item[iv)] For every $a\in R \setminus(\upsilon^{-1}(\infty):M)$, there
is an $a^{'}\in R$ such that \\
$\upsilon((a^{'}a)x)=\upsilon(x)$ for all $x\in M$
\end{enumerate}
 \end{Def}
  \begin{Def}\label{sec:defv6}
   Let $M$ be an $R$\_module where $R$ is a ring, and let $\nu$ be a valuation on
$M$. A representation system of the equivalence relation $\sim_{\nu}$ is called a skeleton of $\nu$.
 \end{Def}
 \begin{Def}\label{sec:defv7}
 A subset $S$ of $M$ is said to be $\nu$-independent if $S\cap \nu^{-1}(\infty)=\phi$, and $\nu(x)\notin \nu(Ry)$ for any pair of distinct elements $x,y\in S$. Here, we adopt the convention that the empty
subset $\phi$ is $\nu$-independent.
 \end{Def}
 \begin{pro} \label{sec:prov1}
 Let $M$ be an $R$\_module where $R$ is a ring, and let \\ $\upsilon:M \rightarrow \Delta$ be a valuation on $M$. Then the following statements are true:
\begin{enumerate}
 \item[i)] If $\nu(x)=\nu(y)$ for $x,y\in M$, then $\nu(ax)=\nu(ay)$ for all $a\in R$;
 \item[ii)] $\nu(-x)=\nu(x)$ for all $x\in M$;
 \item[iii)] If $\nu(x)\neq \nu(y)$, then $\nu(x+y)=min \lbrace \nu(x),\nu(y) \rbrace$;
 \item[iv)] If $\nu(az)=\nu(bz)$ for some $a,b\in R$ and $z\in M\setminus \nu^{-1}(\infty)$,then\\ $\nu(ax)=\nu(bx)$ for all $x\in M$;
 \item[v)] If $\nu(az)<\nu(bz)$ for some $a,b\in R$ and $z\in M$, then  $\nu(ax)<\nu(bx)$ for all $x\in M\setminus \nu^{-1}(\infty)$;
 \item[vi)] The core $\nu^{-1}$ of $\nu$ is prime submodule of $M$;
 \item[vii)] The following subsets constitute a valuation pair of $R$ with core \\
 $(M:\nu^{-1}(\infty))$:
 \begin{equation*}
 A_\nu=\lbrace a\in A \vert \nu(ax)\geq \nu(x) ~ for ~ all ~ x\in M
 \rbrace ,
 \end{equation*}
 \begin{equation*}
 P_\nu=\lbrace a\in A \vert \nu(ax)\geq \nu(x)
  ~ for ~ all ~  x\in M\setminus \nu^{-1}(\infty)\rbrace
 \end{equation*}
 \end{enumerate}
 \end{pro}
 \begin{proof}
see proposition 1.1 \cite{6}
\end{proof}

 \begin{Def}\label{sec:defv8}
 The pair $(A_{\nu},P_{\nu})$ as in Proposition (\ref{sec:prov1}) is called the valuation pair
of $R$ induced by $\nu$ or the induced valuation pair of $\nu$.
\end{Def}

 \section{Valuation derived from filtered module}
\paragraph*{}In this section we use the four following lemmas for showing the existence of valuation on filtered module. Let $R$ be a ring with unit and $R$ a filtered ring with
filtration $\left\{R_{n} \right\}_{n>0} $ and $M$ be filtered
$R$\_module with filtration $\left\{M_{n} \right\}_{n>0} $.
 \begin{lem}\label{sec:lemv31}
 Let $M$ be filtered $R$\_module with filtration $\left\{M_{n} \right\}_{n>0} $. Now we define $\nu :M\to\mathbb{Z}$ such that  for every $t \in M$ and $\nu (t)=\min \left\{i\left|t \in M_{i} \backslash M_{i+1} \right. \right\}$. Then  for all $x,y \in M$ we have $\nu(x+y)\geq min\lbrace \nu(x),\nu(u) \rbrace$.
  \end{lem}
  \begin{proof}
 For any $x,y \in M$ such that $\nu (x )=i$ also $\nu (y )=j$, and $\nu(x+y )=k$, so we have $x+y \in M_{k} \backslash M_{k+1} $. Without losing the generality, let $i<j$ so $M_{j} \subset M_{i} $ hence $y \in R_{i} $. Now if $k<i$, then $k+1\le i$ and $M_{i} \subset M_{k+1} $ so $x+y \in M_{i} \subset M_{k+1} $ it is contradiction. Hence $k\ge i$ and so we have $\nu(x+y )\geq \min \left\{\nu(x),\nu(y)\right\}$.\\
 \end{proof}

 \begin{lem}\label{sec:lemv32}
 Let $M$ be filtered $R$\_module with filtration $\left\{M_{n} \right\}_{n>0} $. Now we define $\nu$ as lemma(\ref{sec:lemv31}). If $\upsilon(y)\leq \upsilon(x)$, $x,y \in M$, then $\upsilon(ay)\leq \upsilon(ax)$ for all $a\in R$;
 \end{lem}
 \begin{proof}
 Let $\nu(x)=i$ and $\nu(y)=j$, since $\nu(x)\geq \nu(y)$ then $M_{j}\supseteq M_{i}$. Since $R$ is filtered ring, there exists $k\in \mathbb{Z}$ such that $a\in R_{k}$ so
 \begin{equation*}
 ax\in R_{k}M_{i}\subseteq M_{k+i}
  \end{equation*}
 \begin{equation*}
 ay\in R_{k}M_{j}\subseteq M_{k+j}
  \end{equation*}
 we have $i+k \geq j+k$ by $i\geq j$, then $\nu(ax)\geq \nu(ay)$ for all $a\in
 R$.
 \end{proof}
  \begin{lem}\label{sec:lemv33}
 Let $M$ be filtered $R$\_module with filtration $\left\{M_{n} \right\}_{n>0} $. Now we define $\nu$ as lemma(\ref{sec:lemv31}). Put $\upsilon^{-1}:=\{x \in M | \upsilon(x)=\infty\}$.
If $\upsilon(az)\leq \upsilon(bz)$, where $a,b\in R$, and $z\in
M\setminus\upsilon^{-1}(\infty)$, then $\upsilon(ax)\leq
\upsilon(bx)$ for all $x\in M$.
 \end{lem}
 \begin{proof}
Since  $a,b\in R$ and $z\in M$ then there exist $i,j,k\in
\mathbb{Z}$ such that \\ $a\in R_{i}$ , $b\in R_{j}$ and $z\in
M_{k}$ hence
\begin{equation*}
az\in R_{i}M_{k}\subseteq M_{i+k}
\end{equation*}
\begin{equation*}
bz\in R_{j}M_{k}\subseteq M_{j+k}
\end{equation*}
Now if $\nu(az)\leq \nu(bz)$ then
\begin{equation*}
k+i\leq k+j \Longrightarrow i\leq j \Longrightarrow R_{j}\subseteq R_{i}
\end{equation*}
So we have $\nu(ax)\leq \nu(bx)$ for all $x\in M$
 \end{proof}
  \begin{lem}\label{sec:lemv34}
 Let $M$ be filtered $R$\_module with filtration $\left\{M_{n} \right\}_{n>0} $. Now we define $\nu$ as lemma(\ref{sec:lemv31}). For every $a\in R \setminus(\upsilon^{-1}(\infty):M)$, there
is an $a^{'}\in R$ such that $\upsilon((a^{'}a)x)=\upsilon(x)$
for all $x\in M$.
 \end{lem}
 \begin{proof}
 Let  $x\in \nu^{-1}(\infty)$ then for all $a^{'},a\in R$ $\upsilon((a^{'}a)x)=\upsilon(x)=\infty$.\\
  Now let $x \notin \nu^{-1}(\infty)$ and for all $a^{'}\in R$ we have $\nu((a^{'}a)x)\neq \nu(x)$. So
 if $a^{'}\in R\setminus(\nu^{-1}(\infty):M)$, then $a^{'}a\in R\setminus (\nu^{-1}(\infty):M)$ and hence $\nu((a^{'}a)x)\neq\infty$.\\
Let $a\in R_{k}$ , $a^{'}\in R_{k^{'}}$ and $x\in M_{i}$, then $a^{'}a\in R_{k+k^{'}}$ so $(a^{'}a)x\in M_{i+k+k^{'}}$.\\
We may have one of following conditions:
\begin{enumerate}
\item[1)] $\nu((a^{'}a)x)< \nu(x)$.
\item[2)] $\nu(x)<\nu((a^{'}a)x)$\\
\end{enumerate} Now if we have (1) then $i+k+k^{'} < i$, it is
contradiction . \\
Consequently $a^{'}\in R_{k^{'}}$ and $a\in R_{k}$ for $k\in
\mathbb{Z}$ then
\begin{equation*}
a^{'}a\in R_{k^{'}+k} \Longrightarrow (a^{'}a)x\in
R_{k+k^{'}}M_{i}\subseteq M_{i+k^{'}+k}.
\end{equation*}
Since $M_{k^{'}+k+i}\subseteq M_{i} $ hence $(a^{'}a)x\in M_{i}$.
So we have $\nu((a^{'}a)x)<i$ therefore $\nu(x)>\nu((a^{'}a)x)$,
it
is contradiction with (2). By now we have \\
$\nu(x)=\nu((a^{'}a)x)$.
 \end{proof}
 \begin{theo}\label{sec:thv1}
 Let $R$ be a filtered ring with filtration $\left\{R_{n} \right\}_{n>0} $, and $M$ be a filtered $R$\_module with filtration $\left\{M_{n} \right\}_{n>0} $. Now we define $\nu :M\to\mathbb{Z}$ such that  for every $t \in M$ and $\nu (t)=\min \left\{i\left|t \in M_{i} \backslash M_{h+1} \right. \right\}$. Then $\nu $ is a valuation on $M$.
 \end{theo}
   \begin{proof}
\begin{enumerate}
 \item[i)] By lemma (\ref{sec:lemv31}) we have For any $x,y\in M$, $\upsilon(x+y)\geq min\{\upsilon(x),\upsilon(y)\}$;
 \item[ii)] We have  If $\upsilon(x)\leq \upsilon(y)$, $x,y \in M$, then $\upsilon(ax)\leq
\upsilon(ay)$ for all $a\in R$ by lemma(\ref{sec:lemv32});
\item[iii)] Put $\upsilon^{-1}:=\{x \in M | \upsilon(x)=\infty\}$.
If $\upsilon(az)\leq \upsilon(bz)$, where $a,b\in R$, and $z\in
M\setminus\upsilon^{-1}(\infty)$, then then by lemma (\ref{sec:lemv33}) $\upsilon(ax)\leq
\upsilon(ay)$ for all $x\in M$;
\item[iv)] For every $a\in R \setminus(\upsilon^{-1}(\infty):M)$, then by lemma(\ref{sec:lemv34}) there
is an $a^{'}\in R$ such that
$\upsilon((a^{'}a)x)=\upsilon(x)$ for all $x\in M$.\\
\end{enumerate}
 So by definition(\ref{sec:defv5}) $\nu$ is a valuation on$M$ if has those conditions.
 \end{proof}

 \begin{cro}\label{sec:crov11}
 If $M$ be a filtered $R$\_module, then $\nu :M\to\mathbb{Z}$ has
 all of properties that explained in Proposition(\ref{sec:prov1}).
 \end{cro}
 \begin{pro}\label{sec:prov11}
 I $R$ is a strongly filtered ring and $M$ is a strongly filtered
 $R$\_module and there exist valuation $\nu :M\to\mathbb{Z}$ on $M$, then R should be a trivial filtered $R$\_module.
 \end{pro}
 \begin{proof}
 By definition(\ref{sec:defv5})(iv) and theorem(\ref{sec:thv1}) we
 have for every $a\in R \setminus(\upsilon^{-1}(\infty):M)$, there
is an $a^{'}\in R$ such that $\upsilon((a^{'}a)x)=\upsilon(x)$.
Now if $\nu(a)=i$ , $\nu(a_{'})=j$ and $\nu(x)=k$ then $i+j+k=k$
so $i+j=0$, consequently $R_{i}=R$ for every $i>0$.
\end{proof}

\begin{pro}\label{sec:prov2}
 Let $M$ be an $R$\_module, where $R$ is a ring. Then there is a valuation
on $M$, if and only if there exists a prime ideal $P$ of $R$ such
that $PM_{P}\neq M_{P}$, where $M_{P}$ is the localization of $M$
at $P$.
\end{pro}
\begin{proof}
see (Proposition 1.3 \cite{6})
\end{proof}

\begin{cro}\label{sec:crov1}
 Let $M$ be an filtered $R$\_module, where $R$ is a filtered ring. Then there exists a prime ideal $P$ of $R$ such that $PM_{P}\neq M_{P}$, where $M_{P}$ is the localization of $M$ at $P$.
\end{cro}
\begin{proof}
By theorem(\ref{sec:thv1}) there is an valuation on $M$, then by proposition(\ref{sec:prov2}) there exists a prime ideal $P$ of $R$ such that $PM_{P}\neq M_{P}$, where $M_{P}$ is the localization of $M$ at $P$.
\end{proof}
\begin{cro}\label{sec:crov3}
Let $M$ be an filtered $R$\_module, where $R$ is a filtered ring.
Then there is a skeleton on $M$.
\end{cro}
\begin{proof}
By theorem(\ref{sec:thv1}) there is a valuation on $M$,then by
definition(\ref{sec:defv6}) we have there is a skeleton on $M$.
\end{proof}
\begin{pro}\label{prov3}
Let $M$ be an filtered $R$\_module where $R$ is a filtered  ring,
and $\nu$ a valuation on $M$. If $\Lambda$ is a skeleton of
$\nu$, then the following conditions are satisfied:
\begin{enumerate}
\item[i)] $\Lambda$ is a $\nu$-independent subset of $M$;
\item[ii)] For every $x\in M\setminus\nu^{-1}(\infty)$, there exists a unique $\lambda\in \Lambda$ such that \\ $\nu(x)=\nu(R\lambda)$.
\end{enumerate}
\end{pro}
\begin{proof}
By corollary(\ref{sec:crov3}) $\Lambda$ is a skeleton of $\nu$
and by proposition(1.4, \cite{6}) we have the above conditions.
\end{proof}
\begin{pro}\label{sec;prov4}
Let $M$ be an filtered $R$\_module where $R$ is a filtered  ring,
and $\nu$ a valuation on $M$. If $\Lambda$ is a skeleton of
$\nu$. If $a_{1}\lambda_{1}+\cdots+a_{n}\lambda_{n}=0$ where
$a_{1},\cdots a_{n}\in R$ and
$\lambda_{1}\cdots\lambda_{n}\in\Lambda$ are mutually distinct,
then \\ $a_{i}\in(\nu^{-1}(\infty):M)$,$i=1,\cdots ,n$.
\end{pro}
\begin{proof}
By corollary(\ref{sec:crov3}) $\Lambda$ is a skeleton of $\nu$
and by proposition(1.5, \cite{6}) we have If
$a_{1}\lambda_{1}+\cdots+a_{n}\lambda_{n}=0$ where $a_{1},\cdots
a_{n}\in R$ and $\lambda_{1}\cdots\lambda_{n}\in\Lambda$ are
mutually distinct, then
$a_{i}\in(\nu^{-1}(\infty):M)$,$i=1,\cdots ,n$.
\end{proof}


\begin{thebibliography}{99}
\bibitem{1}J. Alajbegovic,   Approximation theorems for Manis valuations with the inverse, (1984).
property. Comm. Algebra 12:1399–1417.
\bibitem{2}M. F. Atiyah, I. G. Macdonald,   Introduction to Commutative Algebra. Massachusetts:(1969).
Addison-Wesley.N.
\bibitem{3}O. Endler,   Valuation Theory. New York: Springer-Verlag.Algebras, Rings and Modules  by Michiel Hazewinkel CWI,Amsterdam, The Netherlands Nadiya Gubareni Technical University of Czêstochowa,Poland and V.V. KirichenkoKiev Taras Shevchenko University,Kiev, Ukraine KLUWER.(1972).
\bibitem{4}Fuchs, L. Partially Ordered Algebraic Systems. New York: Pergammon Press.S.(1963).
\bibitem{5} Gopalakrishnan, Commutative algebra,oxonian press,1983.
\bibitem{6}Z. Guangxing, Valuations on a Module, Communications in Algebra, 2341-2356, 2007.
\bibitem{7} Huckaba, J. A. Commutative Rings with Zero Divisors. New York: Marcel Dekker, Inc.(1988).

  \bibitem{8}T.Y. Lam , A First Course in Noncommutative Rings,Springer-Verlag ,1991.
  \bibitem{9}C. P. Lu,   Spectra of modules. Comm. Algebra 175:3741–3752. (1995).
\bibitem{10}O. F. G. Schilling,   The Theory of Valuations. New York: Amer. Math. Soc. (1952).
\end{thebibliography}
\end{document}